\documentclass{amsart}

\usepackage{amsthm}
\usepackage[leqno]{amsmath}
\usepackage[all]{xy} \SelectTips{eu}{}
\usepackage{latexsym,amsfonts,amssymb}
\usepackage{hyperref}
\usepackage{mathptmx}
\usepackage{nicefrac}

\newcommand{\numberseries}{\bfseries}   

\newlength{\thmtopspace}                
\newlength{\thmbotspace}                
\newlength{\thmheadspace}               
\newlength{\thmindent}                  

\newtheoremstyle{fixed bf head,slanted body}
                {\thmtopspace}{\thmbotspace}{\slshape}
                {\thmindent}{\bfseries}{.}{\thmheadspace}
                {{\numberseries \thmnumber{#2\;}}\thmname{#1}\thmnote{ (#3)}}

\newtheoremstyle{variable bf head,slanted body}
                {\thmtopspace}{\thmbotspace}{\slshape}
                {\thmindent}{\bfseries}{.}{\thmheadspace}
                {{\numberseries \thmnumber{#2\;}}\thmname{#1}\thmnote{ #3}}

\newtheoremstyle{fixed bf head,upright body}
                {\thmtopspace}{\thmbotspace}{\upshape}
                {\thmindent}{\bfseries}{.}{\thmheadspace}
                {{\numberseries \thmnumber{#2\;}}\thmname{#1}\thmnote{ (#3)}}

\newtheoremstyle{numbered paragraph}
                {\thmtopspace}{\thmbotspace}{\upshape}
                {\thmindent}{\upshape}{}{\thmheadspace}
                {{\numberseries \thmnumber{#2.}}}

\theoremstyle{fixed bf head,slanted body}
\newtheorem{res}{}[section]
          \newtheorem*{thm*}{Theorem}
\newtheorem{prp}[res]{Proposition}      \newtheorem*{prp*}{Proposition}
        \newtheorem*{cor*}{Corollary}
\newtheorem{lem}[res]{Lemma}            \newtheorem*{lem*}{Lemma}

\theoremstyle{variable bf head,slanted body}
     \newtheorem*{introthm*}{Theorem}

\theoremstyle{fixed bf head,upright body}
\newtheorem{stp}[res]{Setup}            \newtheorem*{stp*}{Setup}
\newtheorem{dfn}[res]{Definition}       \newtheorem*{dfn*}{Definition}
     \newtheorem*{con*}{Construction}
      \newtheorem*{obs*}{Observation}
\newtheorem{rmk}[res]{Remark}           \newtheorem*{rmk*}{Remark}
\newtheorem{exa}[res]{Example}          \newtheorem*{exa*}{Example}
         \newtheorem*{qst*}{Question}

\theoremstyle{numbered paragraph}
\newtheorem{ipg}[res]{}

\newlength{\thmlistleft}        
\newlength{\thmlistright}       
\newlength{\thmlistpartopsep}   
\newlength{\thmlisttopsep}      
\newlength{\thmlistparsep}      
\newlength{\thmlistitemsep}     

\newcounter{eqc} 
\newenvironment{eqc}{\begin{list}{\upshape (\textit{\roman{eqc}})}%
    {\usecounter{eqc}%
      \setlength{\leftmargin}{\thmlistleft}%
      \setlength{\labelwidth}{\thmlistleft}%
      \setlength{\rightmargin}{\thmlistright}%
      \setlength{\partopsep}{\thmlistpartopsep}%
      \setlength{\topsep}{\thmlisttopsep}%
      \setlength{\parsep}{\thmlistparsep}%
      \setlength{\itemsep}{\thmlistitemsep}}}%
  {\end{list}}%

\newcommand{\eqclbl}[1]{{\upshape(\textit{#1})}}

\newcounter{prt}
\newenvironment{prt}{\begin{list}{\upshape (\alph{prt})}%
    {\usecounter{prt}%
      \setlength{\leftmargin}{\thmlistleft}%
      \setlength{\labelwidth}{\thmlistleft}%
      \setlength{\rightmargin}{\thmlistright}%
      \setlength{\partopsep}{\thmlistpartopsep}%
      \setlength{\topsep}{\thmlisttopsep}%
      \setlength{\parsep}{\thmlistparsep}%
      \setlength{\itemsep}{\thmlistitemsep}}}%
  {\end{list}}%

\newcommand{\prtlbl}[1]{{\upshape(#1)}}

  \newcommand{\proofofimp}[3][:]{\mbox{\eqclbl{#2}$\!\implies\!$\eqclbl{#3}#1}}

\newcommand{\pgref}[1]{\ref{#1}}

\newcommand{\prpref}[2][Proposition~]{#1\pgref{prp:#2}}
\newcommand{\lemref}[2][Lemma~]{#1\pgref{lem:#2}}

\newcommand{\dfnref}[2][Definition~]{#1\pgref{dfn:#2}}
\newcommand{\exaref}[2][Example~]{#1\pgref{exa:#2}}

\newcommand{\secref}[2][Section~]{#1\ref{sec:#2}}

\renewcommand{\eqref}[1]{(\pgref{eq:#1})}

\makeatletter
\def\@nobreak@#1{\mathchoice%
  {\nobreakdef@\displaystyle\f@size{#1}}%
  {\nobreakdef@\nobreakstyle\tf@size{\firstchoice@false #1}}%
  {\nobreakdef@\nobreakstyle\sf@size{\firstchoice@false #1}}%
  {\nobreakdef@\nobreakstyle\ssf@size{\firstchoice@false #1}}%
  \check@mathfonts}%
\def\nobreakdef@#1#2#3{\hbox{{%
                    \everymath{#1}%
                    \let\f@size#2\selectfont%
                    #3}}}%
\makeatother

\DeclareSymbolFont{usualmathcal}{OMS}{cmsy}{m}{n}
\DeclareSymbolFontAlphabet{\mathcal}{usualmathcal}

\DeclareSymbolFont{letters}{OML}{txmi}{m}{it}

\DeclareMathSymbol{\alpha}{\mathord}{letters}{"0B}
\DeclareMathSymbol{\beta}{\mathord}{letters}{"0C}
\DeclareMathSymbol{\gamma}{\mathord}{letters}{"0D}
\DeclareMathSymbol{\delta}{\mathord}{letters}{"0E}
\DeclareMathSymbol{\epsilon}{\mathord}{letters}{"0F}
\DeclareMathSymbol{\zeta}{\mathord}{letters}{"10}
\DeclareMathSymbol{\eta}{\mathord}{letters}{"11}
\DeclareMathSymbol{\theta}{\mathord}{letters}{"12}
\DeclareMathSymbol{\iota}{\mathord}{letters}{"13}
\DeclareMathSymbol{\kappa}{\mathord}{letters}{"14}
\DeclareMathSymbol{\lambda}{\mathord}{letters}{"15}
\DeclareMathSymbol{\mu}{\mathord}{letters}{"16}
\DeclareMathSymbol{\nu}{\mathord}{letters}{"17}
\DeclareMathSymbol{\xi}{\mathord}{letters}{"18}
\DeclareMathSymbol{\pi}{\mathord}{letters}{"19}
\DeclareMathSymbol{\rho}{\mathord}{letters}{"1A}
\DeclareMathSymbol{\sigma}{\mathord}{letters}{"1B}
\DeclareMathSymbol{\tau}{\mathord}{letters}{"1C}
\DeclareMathSymbol{\upsilon}{\mathord}{letters}{"1D}
\DeclareMathSymbol{\phi}{\mathord}{letters}{"1E}
\DeclareMathSymbol{\chi}{\mathord}{letters}{"1F}
\DeclareMathSymbol{\psi}{\mathord}{letters}{"20}
\DeclareMathSymbol{\omega}{\mathord}{letters}{"21}
\DeclareMathSymbol{\varepsilon}{\mathord}{letters}{"22}
\DeclareMathSymbol{\vartheta}{\mathord}{letters}{"23}
\DeclareMathSymbol{\varpi}{\mathord}{letters}{"24}
\DeclareMathSymbol{\varrho}{\mathord}{letters}{"25}
\DeclareMathSymbol{\varsigma}{\mathord}{letters}{"26}
\DeclareMathSymbol{\varphi}{\mathord}{letters}{"27}

\DeclareMathSymbol{\Gamma}{\mathord}{letters}{"00}
\DeclareMathSymbol{\Delta}{\mathord}{letters}{"01}
\DeclareMathSymbol{\Theta}{\mathord}{letters}{"02}
\DeclareMathSymbol{\Lambda}{\mathord}{letters}{"03}
\DeclareMathSymbol{\Xi}{\mathord}{letters}{"04}
\DeclareMathSymbol{\Pi}{\mathord}{letters}{"05}
\DeclareMathSymbol{\Sigma}{\mathord}{letters}{"06}
\DeclareMathSymbol{\Upsilon}{\mathord}{letters}{"07}
\DeclareMathSymbol{\Phi}{\mathord}{letters}{"08}
\DeclareMathSymbol{\Psi}{\mathord}{letters}{"09}
\DeclareMathSymbol{\Omega}{\mathord}{letters}{"0A}

\DeclareMathSymbol{\upGamma}{\mathalpha}{operators}{"00}
\DeclareMathSymbol{\upDelta}{\mathalpha}{operators}{"01}
\DeclareMathSymbol{\upTheta}{\mathalpha}{operators}{"02}
\DeclareMathSymbol{\upLambda}{\mathalpha}{operators}{"03}
\DeclareMathSymbol{\upXi}{\mathalpha}{operators}{"04}
\DeclareMathSymbol{\upPi}{\mathalpha}{operators}{"05}
\DeclareMathSymbol{\upSigma}{\mathalpha}{operators}{"06}
\DeclareMathSymbol{\upUpsilon}{\mathalpha}{operators}{"07}
\DeclareMathSymbol{\upPhi}{\mathalpha}{operators}{"08}
\DeclareMathSymbol{\upPsi}{\mathalpha}{operators}{"09}
\DeclareMathSymbol{\upOmega}{\mathalpha}{operators}{"0A}

\newcommand{\MCM}{\mathsf{MCM}}
\renewcommand{\mod}{\mathsf{mod}}
\newcommand{\da}[1]{#1^\dagger}
\newcommand{\dda}[1]{#1^{\dagger\dagger}}
\newcommand{\ddda}[1]{#1^{\dagger\dagger\dagger}}
\newcommand{\Coker}[1]{\operatorname{Coker}#1}
\newcommand{\Ker}[1]{\operatorname{Ker}#1}
\renewcommand{\Im}[1]{\operatorname{Im}#1}
\newcommand{\Hom}[3][R]{\operatorname{Hom}_{#1}(#2,#3)}
\newcommand{\Ext}[4][R]{\operatorname{Ext}_{#1}^{#2}(#3,#4)}

\newcommand{\depth}[2][R]{\operatorname{depth}_{#1}#2}
\newcommand{\dual}{\upOmega}
\newcommand{\m}{\mathfrak{m}}
\newcommand{\n}{\mathfrak{n}}

\renewcommand{\a}{\mathfrak{a}}

\newcommand{\Syz}[3][\mathcal{A}]{\operatorname{Syz}^{#1}_{#2}(#3)}
\newcommand{\mSyz}[3][\mathcal{A}]{\operatorname{syz}^{#1}_{#2}(#3)}
\newcommand{\Cosyz}[3][\mathcal{A}]{\operatorname{Cosyz}_{#1}^{#2}(#3)}
\newcommand{\mCosyz}[3][\mathcal{A}]{\operatorname{cosyz}_{#1}^{#2}(#3)}

\begin{document}

\title{Approximations by maximal Cohen--Macaulay modules}

\author{Henrik Holm}

\address{University of Copenhagen, 2100 Copenhagen {\O}, Denmark}
 
\email{holm@math.ku.dk}

\urladdr{http://www.math.ku.dk/\~{}holm/}


\keywords{Cosyzygy; envelope; maximal Cohen--Macaulay module; special preenvelope; uni\-que lifting property.}

\subjclass[2010]{13C14; 13D05}


\begin{abstract}
  Auslander and Buchweitz have proved that every finitely generated module over a Cohen--Macaulay (CM) ring with a dualizing module admits a so-called maximal CM approximation. In terms of relative homological algebra, this means that every finitely generated module has a special maximal CM precover. In this paper, we prove the existence of special maximal CM preenvelopes and, in the case where the ground ring is henselian, of maximal CM envelopes. We also characterize the rings over which every finitely generated module has a maximal CM envelope with the unique lifting property. Finally, we show that cosyzygies with respect to the class of maximal CM modules must eventually be maximal CM, and we compute some examples.  
\end{abstract}

\maketitle

\section{Introduction}

Let $R$ be a commutative noetherian local Cohen--Macaulay (CM) ring with a dualizing module $\dual$ and denote by $\MCM$ the class of maximal CM $R$-modules. Auslander and Buchweitz construct in \cite[Thm.~A]{MAsROB89} a \emph{maximal CM approximation} for every finitely generated $R$-module $M$, that is, a short exact sequence,
\begin{displaymath}
    0 \longrightarrow I \longrightarrow X \stackrel{\pi}{\longrightarrow} M \longrightarrow 0
\end{displaymath}
where $X$ belongs to $\MCM$ and $I$ has finite injective dimension. By a result of Ischebeck~\cite{FIs69} one has $\Ext{1}{Y}{I}=0$ for all $Y$ in $\MCM$, so in terms of relative homological algebra, this means that the homomorphism $\pi \colon X \twoheadrightarrow M$ is a \emph{special $\MCM$-precover} of $M$. A result of Takahashi \cite[Cor.~2.5]{RTk05a} shows that if $R$ is henselian (for example, if $R$ is complete), then every $\MCM$-precover can be ``refined'' to an \emph{$\MCM$-cover}. This result of Takahashi follows from Prop.~2.4 in \emph{loc.~cit.}, which the author contributes to Yoshino \cite[Lem.~2.2]{YY93} (written in Japanese). We summarize these results in the following theorem.

\begin{introthm*}[(Auslander and Buchweitz \cite{MAsROB89}, Takahashi \cite{RTk05a}, and Yoshino \cite{YY93})]
  \quad
  \begin{prt}
  \item Every finitely generated $R$-module has a special $\MCM$-precover (also called a 
special right $\MCM$-approximation). 
  \item If $R$ is henselian, then every finitely generated $R$-module has an $\MCM$-cover (also called a minimal right $\MCM$-approximation).
  \end{prt}
\end{introthm*}

This paper is concerned with the existence and the construction of special $\MCM$-preen\-ve\-lopes and $\MCM$-envelopes of finitely generated modules. Our first main result, which is proved in \secref{MCM-envelopes}, is the following ``dual'' of the theorem above.

\begin{introthm*}[A]
  The following assertions hold.
  \begin{prt}
  \item
   Every finitely generated $R$-module $M$ has a special $\MCM$-preenvelope (also called a special left $\MCM$-approximation).

\item If $R$ is henselian, then every finitely generated $R$-module has an $\MCM$-envelope (also called a minimal left $\MCM$-approximation).
  \end{prt}
  Moreover, every special $\MCM$-preenvelope, in particular, every $\MCM$-envelope \mbox{$\mu \colon M \to X$} of a finitely generated $R$-module $M$ has the property that $\Hom{\Coker{\mu}}{\dual}$ has finite injective dimension. 
\end{introthm*}

We mention that \cite[Thm.~C]{Holm14} shows the existence of (non-special!) $\MCM$-preenve\-lopes, but its proof is not constructive: It it a consequence of an abstract result by Crawley-Boevey \cite[Thm.~(4.2)]{WCB94} combined with the fact---also proved in \cite{Holm14}---that the direct limit closure of $\MCM$ is closed under products. Theorem~A above is not only stronger than \cite[Thm.~C]{Holm14}; our proof---which is modelled on that of \cite[Thm.~1.6]{HJ11}---also shows how (special) $\MCM$-(pre)envelopes can be constructed from (special) $\MCM$-(pre)covers.

In \secref{Examples} we compute the $\MCM$-envelope of some specific modules. In \secref{ulp} we turn our attention to $\MCM$-envelopes with the \emph{unique lifting property}, and we characterize the rings over which every finitely generated module admits such an envelope:

\begin{introthm*}[B]
  The following conditions are equivalent.
  \begin{eqc}
  \item For every finitely generated $R$-module $M$, the module $\Hom{M}{\dual}$ is maximal CM.
  
  \item The Krull dimension of $R$ is $\,\leqslant 2$.  

  \item The inclusion functor $\MCM \hookrightarrow \mod$ has a left adjoint.
      
  \item Every finitely generated $R$-module has an $\MCM$-envelope with the unique lifting property.
  \end{eqc}
\end{introthm*}

From a homological point of view, maximal CM modules are interesting because every module can be finitely resolved by such modules. More precisely, if $d$ denotes the Krull dimension of the CM ring $R$, and if $M$ is any finitely generated $R$-module with a resolution
\begin{displaymath}
\cdots \longrightarrow X_d \longrightarrow X_{d-1} \longrightarrow X_{d-2} \longrightarrow \cdots \longrightarrow X_1 \longrightarrow X_0 \longrightarrow M \longrightarrow 0
\end{displaymath}
by finitely generated free $R$-modules $X_0, X_1,\ldots$, then the $n^\mathrm{th}$ syzygy of $M$, i.e.~the module $\Syz[]{n}{M}=\Ker(X_{n-1} \to X_{n-2})$, is maximal CM for every $n \geqslant d$. Actually, the same conclusion holds if $X_0, X_1,\ldots$ are just assumed to be maximal CM (but not necessarily free); this well-known fact follows from the behaviour of depth in short exact sequences; see Bruns and Herzog \cite[Prop.~1.2.9]{BruHer} or \lemref{depth}. Given a finitely generated $R$-module $M$, one can \textsl{not} always construct an \textsl{exact} sequence
\begin{displaymath}
\tag{\text{$*$}}
0 \longrightarrow M \longrightarrow X^0 \longrightarrow X^1 \longrightarrow \cdots
\end{displaymath}
where $X^0,X^1,\ldots$ are maximal CM; however, there is a canonical way to construct a \mbox{\textsl{complex}} of the form ($*$). Theorem~A guarantees the existence of $\MCM$-preenvelopes, which makes the following construction possible: Take an $\MCM$-preenvelope $\mu^0 \colon M \to X^0$ of $M$ and set $C^1 = \Coker{\mu^0}$; take an $\MCM$-preenvelope $\mu^1 \colon C^1 \to X^1$ of $C^1$ and set $C^2 = \Coker{\mu^1}$;~etc. The hereby constructed complex ($*$)\,---\,which is called a \emph{proper $\MCM$-coresolution} or an \emph{$\MCM$-resolvent} of $M$\,---\,is not necessarily exact, but it becomes exact if one applies the functor $\Hom{-}{Y}$ to it for any $Y$ in $\MCM$. The module $C^n = \Coker{(X^{n-2} \to X^{n-1})}$ is called the \emph{$n^\mathrm{th}$ cosyzygy of $M$ with respect to $\MCM$}, and it is denoted by $\Cosyz[\MCM]{n}{M}$. In \secref{cosyz} we prove that such cosyzygies must eventually be maximal CM:

\enlargethispage{4.5ex}

\begin{introthm*}[C]
Let $M$ be a finitely generated $R$-module. For every $n \geqslant d$ the $n^\mathrm{th}$ co\-sy\-zy\-gy, $\Cosyz[\MCM]{n}{M}$, of $M$ with respect to $\MCM$ is maximal CM.
\end{introthm*}

\section{Preliminaries}

\begin{stp}
  \label{stp:setup}
  Throughout, $(R,\m,k)$ is a commutative noetherian local CM ring of Krull dimension $d$. It is assumed that $R$ has a dualizing (or canonical) module~$\dual$. 
\end{stp}

Let $M$ be a finitely generated $R$-module. The \emph{depth} of $M$ is the number
\begin{displaymath}
  \depth{M} \,=\, \inf\{i\,|\,\Ext{i}{k}{M} \neq 0\} \,\in\, \mathbb{N}_0 \cup \{\infty\}\;;
\end{displaymath}
see \cite[Def.~1.2.6 and 1.2.7]{BruHer}. If $M \neq 0$, then $\depth{M}$ is the common length of a maximal $M$-regular sequence (in $\m$). The depth can also be computed from the dualizing module:
\begin{displaymath}
  \depth{M} \,=\, d-\sup\{i\,|\,\Ext{i}{M}{\dual} \neq 0\}\;;
\end{displaymath}
see \cite[Cor.~3.5.11]{BruHer}. One calls $M$ for \emph{maximal CM} if $\depth{M} \geqslant d$, that is, $\Ext{i}{M}{\dual}=0$ for all $i>0$. The category of all such $R$-modules is denoted by $\MCM$.  The category of all finitely generated $R$-modules is denoted by $\mod$. 

We recall a few notions from relative homological algebra.

\begin{dfn}
  \label{dfn:preenvelope}
  Let $\mathcal{A}$ be a full subcategory of an abelian category $\mathcal{M}$ (e.g.~$\mathcal{M}= \mod$ and $\mathcal{A}=\MCM$), and let $M$ be an object in $\mathcal{M}$. Following Enochs and Jenda \cite[Def.~5.1.1]{rha}, a morphism $\pi \colon A \to M$ with $A \in \mathcal{A}$ is called an \emph{$\mathcal{A}$-precover}
(or a \emph{right $\mathcal{A}$-approximation}) of $M$ if every other morphism $\pi' \colon A' \to M$ with $A' \in \mathcal{A}$ factors through $\pi$, as illustrated below.
\begin{displaymath}
  \xymatrix{
    {} & A' \ar@{-->}[dl] \ar[d]^-{\pi'} \\
    A \ar[r]^-{\pi} & M
   }
\end{displaymath}
A \emph{special $\mathcal{A}$-precover} (or a \emph{special right $\mathcal{A}$-approximation}) is an $\mathcal{A}$-precover $\pi \colon A \to M$ such that $\Ext[\mathcal{M}]{1}{A'}{\Ker{\pi}}=0$ for every $A' \in \mathcal{A}$. An \emph{$\mathcal{A}$-cover} (or a \emph{minimal right $\mathcal{A}$-approxi\-mation}) is an $\mathcal{A}$-precover $\pi$ with the property that every endomorphism $\varphi$ of $A$ that satisfies $\pi\varphi=\pi$ is an automorphism.

The notions of \emph{$\mathcal{A}$-preenvelope} (or \emph{left $\mathcal{A}$-approximation}), \emph{special $\mathcal{A}$-preenvelope} (or \emph{special left $\mathcal{A}$-approximation}), and \emph{$\mathcal{A}$-envelope} (or \emph{minimal left $\mathcal{A}$-approximation}) are categorically dual to the notions defined above. 

By definition, a special $\mathcal{A}$-precover/preenvelope is also an (ordinary) $\mathcal{A}$-precover/pre\-en\-ve\-lope. If $\mathcal{A}$ is closed under extensions in $\mathcal{M}$, then every $\mathcal{A}$-cover/envelope is a special $\mathcal{A}$-precover/preenvelope; this is the content of Wakamatsu's lemma\footnote{\ Wakamatsu's lemma is implicitly in \cite{Wakamatsu88} by Wakamatsu. It is explicitly stated in Auslander and Reiten \cite[lem.~1.3]{MAsIRt91}, but without a proof. It is stated and proved in Xu \cite[lem.~2.1.1 and 2.1.2]{xu}.}.
\end{dfn}

\begin{ipg}
  \label{MCM-duality}
It is well-known that the dualizing module $\dual$ gives rise to a duality on the category of maximal CM modules; more precisely, there is an equivalence of categories:
\begin{displaymath}
  \xymatrix@C=5pc{
    \MCM \ar@<0.6ex>[r]^-{\Hom{-}{\dual}} & \MCM^\mathrm{op}\;. \ar@<0.6ex>[l]^-{\Hom{-}{\dual}} 
  }
\end{displaymath}
We use the shorthand notation $(-)^\dagger$ for the functor $\Hom{-}{\dual}$. For any finitely generated $R$-module $M$ there is a canonical homomorphism $\delta_M \colon M \to M^{\dagger\dagger}$, called the \emph{biduality homomorphism}, which is natural in $M$. Because of the equivalence above, $\delta_M$ is an isomorphism if $M$ belongs to $\MCM$; cf.~\cite[Thm.~3.3.10]{BruHer}.
\end{ipg}

We will need the following result about depth; it is folklore and easily proved\footnote{\ One way to prove \lemref{depth} is by induction on $m$, using the behaviour of depth on short exact sequences recorded in Bruns and Herzog \cite[Prop.~1.2.9]{BruHer}.}.

\begin{lem}
  \label{lem:depth}
Let $m \geqslant 0$ be an integer and let $\,0 \to K_{m} \to X_{m-1} \to \cdots \to X_0 \to M \to 0$ be an exact sequence of finitely generated $R$-modules. If $X_0,\ldots,X_{m-1}$ are maximal CM, then one has $\depth{K_m} \geqslant \min\{d,\depth{M}+m\}$. In particular, if $m \geqslant d$ then the $R$-module $K_m$ is maximal CM. \qed
\end{lem}


\section{Special $\MCM$-preenvelopes and $\MCM$-envelopes}
\label{sec:MCM-envelopes}

In this section, we prove Theorem~A from the Introduction. Our proof follows that of \cite[Thm.~1.6]{HJ11} with some adjustments.

\begin{lem}
  \label{lem:split}
  For every $R$-module $M$, the composition \smash{\mbox{\!$\xymatrix@C=1.5pc{
       \da{M} \ar[r]^-{\delta_{\da{M}}} & 
       \ddda{M} \ar[r]^-{\da{\delta_{M}}} & \da{M}}$\!}} is the identity map on $\da{M}$.
\end{lem}

\begin{proof}
  Let $\varphi$ be an arbitrary element in $\da{M} = \Hom{M}{\dual}$. We must show that the element $(\da{\delta_{M}} \circ \delta_{\da{M}})(\varphi) = \delta_{\da{M}}(\varphi) \circ \delta_M$ is nothing but $\varphi$, that is, we must prove that for every $x \in M$ one has $(\delta_{\da{M}}(\varphi) \circ \delta_M)(x) = \varphi(x)$. The definitions yield
\begin{displaymath}
  (\delta_{\da{M}}(\varphi) \circ \delta_M)(x) 
  \,=\, \delta_{\da{M}}(\varphi)\big(\delta_M(x)\big)
  \,=\, \delta_M(x)(\varphi)
  \,=\, \varphi(x)\,. \qedhere
\end{displaymath}
\end{proof}

\begin{lem}
  \label{lem:Ext}
  For every finitely generated $R$-module $M$, the next conditions are equivalent.
  \begin{eqc}
  \item $\Ext{1}{M}{\dual}=0$ and $\Ext{1}{X}{\da{M}}=0$ for every $X \in \MCM$.
  \item $\Ext{1}{M}{Y}=0$ for every $Y \in \MCM$.
  \end{eqc}
\end{lem}

\begin{proof}
  \proofofimp{i}{ii} Assume \eqclbl{i}. Given any $Y \in \MCM$ we must argue that $\Ext{1}{M}{Y}=0$, i.e.~that every short exact sequence \smash{$\mspace{1mu}0 \to Y \stackrel{\alpha}{\to} E \to M \to 0\mspace{1mu}$} splits. As $\Ext{1}{M}{\dual}=0$, the functor $\da{(-)}$ leaves this sequence exact; in fact, the induced short exact sequence
\begin{displaymath}
   0 \longrightarrow \da{M} \longrightarrow \da{E} \stackrel{\da{\alpha}}{\longrightarrow} \da{Y} \longrightarrow 0
\end{displaymath}
splits as $\da{Y}$ belongs to $\MCM$ and hence $\Ext{1}{\da{Y}}{\da{M}}=0$ by assumption. Let \mbox{$\beta \colon \da{Y} \to \da{E}$} be a right inverse of $\da{\alpha}$. Then $\delta_Y^{-1}\da{\beta}\delta_E \colon E \to Y$ is a left inverse of $\alpha$ since one has
\begin{displaymath}
  \delta_Y^{-1}\da{\beta}\delta_E\alpha 
  \,=\, 
  \delta_Y^{-1}\da{\beta}\dda{\alpha}\delta_Y
  \,=\, 
  \delta_Y^{-1}\da{(\da{\alpha}\beta)}\delta_Y
  \,=\, 
  \delta_Y^{-1}1_{\dda{Y}}\delta_Y
  \,=\, 
  1_{Y}\;.  
\end{displaymath}

\proofofimp{ii}{i} Assume \eqclbl{ii}. This assumption implies that $\Ext{1}{M}{\dual}=0$ since 
$\dual \in \MCM$. Given $X \in \MCM$ we must show that $\Ext{1}{X}{\da{M}}=0$, i.e.~that every short exact sequence \smash{$0 \to \da{M} \stackrel{\alpha}{\to} E \to X \to 0$} splits. Since $X$ is in $\MCM$ we have, in particular, $\Ext{1}{X}{\dual}=0$, so application of the functor $\da{(-)}$ yields another short exact sequence:
\begin{displaymath}
  \tag{\text{$*$}}
   0 \longrightarrow \da{X} \longrightarrow \da{E} \stackrel{\da{\alpha}}{\longrightarrow} \dda{M} \longrightarrow 0\;.
\end{displaymath}
As $\da{X}$ belongs to $\MCM$ we have $\Ext{1}{M}{\da{X}}=0$, so the functor $\Hom{M}{-}$ leaves the sequence ($*$) exact. Surjectivity of $\Hom{M}{\da{\alpha}}$ yields a homomorphism $\beta \colon M \to \da{E}$ with $\da{\alpha}\beta = \delta_M$. It follows that $\da{\beta}\delta_E \colon E \to \da{M}$ is a left inverse of $\alpha$ since one has
\begin{displaymath}
  \da{\beta}\delta_E\alpha 
  \,=\, 
  \da{\beta}\dda{\alpha}\delta_{\da{M}}
  \,=\, 
  \da{(\da{\alpha}\beta)}\delta_{\da{M}}
  \,=\, 
  \da{\delta_M}\mspace{1mu}\delta_{\da{M}}
  \,=\, 
  1_{\da{M}}\;,
\end{displaymath}
where the last equality follows from \lemref{split}.
\end{proof}

\begin{proof}[Proof of Theorem~A]
  We begin by proving the last assertion in the theorem. Let $\mu \colon M \to X$ be any special $\MCM$-preenvelope of $M$. By assumption, we have $\Ext{1}{\Coker{\mu}}{Y}=0$ for every $Y \in \MCM$, so   \lemref{Ext} implies that $\Ext{1}{Z}{\da{(\Coker{\mu})}}=0$ for every $Z \in \MCM$. By Auslander and Buchweitz \cite[Thm.~A]{MAsROB89}, we can take a \emph{hull of finite injective dimension} for the finitely generated module $\da{(\Coker{\mu})}$, that is, a short exact sequence,
\begin{displaymath}
    0 \longrightarrow \da{(\Coker{\mu})} \longrightarrow I \longrightarrow Z \longrightarrow 0\;,
\end{displaymath}
where $I$ has finite injective dimension and $Z$ is maximal CM. As $\Ext{1}{Z}{\da{(\Coker{\mu})}}=0$, this sequence splits, and $\da{(\Coker{\mu})}$ is therefore a direct summand in $I$. Since $I$ has finite injective dimension, so has $\da{(\Coker{\mu})}$.

  To prove parts \prtlbl{a} and \prtlbl{b}, let $M$ be a finitely generated $R$-module and let $\pi \colon Z \to \da{M}$ be a homomorphism with $Z \in \MCM$. We will show that if $\pi$ is a special $\MCM$-precover, respectively, an $\MCM$-cover\footnote{\ By the theorem of Auslander and Buchweitz, Takahashi, and Yoshino mentioned in the Introduction, special $\MCM$-precovers always exist, and $\MCM$-covers exist if $R$ is henselian}, of $\da{M}$ then the homomorphism 
\begin{displaymath}
  \mu := \da{\pi}\delta_M \colon M \longrightarrow \da{Z}
\end{displaymath}
is a special $\MCM$-preenvelope, respectively, an $\MCM$-envelope, of $M$.

First assume that $\pi$ is a special $\MCM$-precover. We begin by proving that $\mu$ is an $\MCM$-preenvelope. Note that $\da{Z}$ is in $\MCM$ by \ref{MCM-duality}. We must show that $\Hom{\mu}{Y}$ is surjective for every $Y \in \MCM$. By \ref{MCM-duality} every such $Y$ has the form $Y \cong \da{X}$ for some $X \in \MCM$ (namely for $X=\da{Y}$), so it suffices to show that
$\Hom{\mu}{\da{X}}$ is surjective for every $X \in \MCM$. By definition of $\mu$, the homomorphism $\Hom{\mu}{\da{X}}$ is the composition of the maps
  \begin{displaymath}
    \tag{\text{$*$}}
    \xymatrix@C=4.2pc{
       \Hom{\da{Z}}{\da{X}} \ar[r]^-{\Hom{\da{\pi}}{\da{X}}} & 
       \Hom{\dda{M}}{\da{X}} \ar[r]^-{\Hom{\delta_M}{\da{X}}} & \Hom{M}{\da{X}}
    }.
  \end{displaymath}
Via the ``swap'' isomorphism, see Christensen \cite[(A.2.9)]{lnm}, the homomorphisms in ($*$) are identified with the ones in the top row of the following diagram:
\begin{displaymath}
  \tag{\text{$**$}}
  \begin{gathered}
  \xymatrix@C=4.2pc{
     \Hom{X}{\dda{Z}} \ar[r]^-{\Hom{X}{\dda{\pi}}} & 
     \Hom{X}{\ddda{M}} \ar[r]^-{\Hom{X}{\da{\delta_M}}} & \Hom{X}{\da{M}}
     \\
     \Hom{X}{Z} \ar[u]^-{\Hom{X}{\delta_Z}}_-{\cong} \ar@{->>}[r]^-{\Hom{X}{\pi}}& 
     \Hom{X}{\da{M}} \ar[u]^-{\Hom{X}{\delta_{\da{M}}}} \hspace{2ex} \ar@<-1ex>@{=}[ur]
& 
  }
  \end{gathered}
\end{displaymath}
The left square in ($**$) is commutative since the biduality homomorphism $\delta$ is natural, and the right triangle in ($**$) is commutative by \lemref{split}. The map $\delta_Z$ is an isomorphism since $Z$ is in $\MCM$; and $\Hom{X}{\pi}$ is surjective as $\pi$ is an $\MCM$-precover and $X \in \MCM$. It follows that the composition of the maps in the top row of ($**$), and therefore also the map $\Hom{\mu}{\da{X}}$, is surjective. Thus, $\mu$ is an $\MCM$-preenvelope.

To see that $\mu$ is a special $\MCM$-preenvelope, we must prove that $\Ext{1}{\Coker{\mu}}{Y}=0$ for every $Y \in \MCM$. As the functor $\da{(-)}$ is left exact, $\da{(\Coker{\mu})}$ is isomorphic to $\Ker{(\da{\mu})}$. By definition we have \smash{$\da{\mu} = \da{\delta_M}\dda{\pi}$}, and hence $\da{\mu}$ fits into the commutative diagram:
\begin{displaymath}
  \tag{\text{$*\mspace{-2.5mu}*\mspace{-2.5mu}*$}}
  \begin{gathered}
  \xymatrix@R=1.7pc{
    \dda{Z} \ar[r]^-{\da{\mu}} & \da{M} \\
    \dda{Z} \ar@{=}[u] \ar[r]^-{\dda{\pi}} & \ddda{M} \ar[u]_-{\da{\delta_M}} \\ 
    Z \ar[u]^-{\delta_Z}_-{\cong} \ar[r]^-{\pi} & \da{M} \ar[u]_-{\delta_{\da{M}}}
    \ar@/_2.5pc/[uu]_-{1_{\da{M}} \quad \text{(By \lemref{split})}}
  }
  \end{gathered}
\end{displaymath}
It follows that $\da{\mu}$ and $\pi$ are isomorphic maps, and hence they also have isomorphic kernels, that is, 
$\Ker{(\da{\mu})} \cong \Ker{\pi}$. It follows that $\da{(\Coker{\mu})} \cong \Ker{\pi}$. Since $\pi$ is a special $\MCM$-precover, we now have
\begin{displaymath}
  \Ext{1}{X}{\da{(\Coker{\mu})}} \,\cong\, \Ext{1}{X}{\Ker{\pi}} \,=\, 0 
\end{displaymath}
for every $X \in \MCM$. Thus, to see that $\Ext{1}{\Coker{\mu}}{Y}=0$ for every $Y \in \MCM$, it suffices by \lemref{Ext} to prove that $\Ext{1}{\Coker{\mu}}{\dual}=0$. To this end, set $X=\da{Z} \in \MCM$ and consider the factorization of $\mu \colon M \to \da{Z}=X$ given by
\begin{displaymath}
  \xymatrix@!=0.5pc{
     M \ar@{->>}[dr]_-{\mu_0} \ar[rr]^-{\mu} & & X \\
     {} & \Im{\mu} \ar@{^(->}[ur]_-{\iota} & 
  }
\end{displaymath}
where $\mu_0$ is the corestriction of $\mu$ to its image and $\iota$ is the inclusion map. As $\mu_0$ is surjective and $\da{(-)}$ is left exact, the map $\da{\mu_0}$ is injective. As $\dual \in \MCM$ and $\mu$ is an $\MCM$-preenvelope, the map $\da{\mu} = \Hom{\mu}{\dual}$ is surjective; and hence so is $\da{\mu_0}$ since $\da{\mu} = \da{\mu_0}\,\da{\iota}$. Thus, $\da{\mu_0}$ is an isomorphism. Hence $\da{\iota}$ and $\da{\mu}$ are isomorphic maps, and since $\da{\mu}$ is surjective, so is~$\da{\iota}$. Thus, application of $\da{(-)}$ to \smash{$\mspace{1mu}0 \to \Im{\mu} \stackrel{\iota}{\to} X \to \Coker{\mu} \to 0\mspace{1mu}$} yields an exact sequence,
\begin{displaymath}
  \xymatrix@C=1.5pc{
     \da{X} \ar[r]^-{\da{\iota}} & \da{(\Im{\mu})} \ar[r]^-{0} & \Ext{1}{\Coker{\mu}}{\dual} \ar[r] & \Ext{1}{X}{\dual}=0
  },
\end{displaymath}
which forces $\Ext{1}{\Coker{\mu}}{\dual}=0$, as desired.

Finally, assume that $\pi$ is an $\MCM$-cover. We show that $\mu = \da{\pi}\delta_M$ is an $\MCM$-envelope. We have already seen that $\mu$ is an $\MCM$-preenvelope. To show that it is an envelope, let $\varphi$ be an endomorphism of $\da{Z}$ with $\varphi\mu=\mu$. It follows that $\da{\mu}\da{\varphi}=\da{\mu}$. The diagram \mbox{($*\mspace{-2.5mu}*\mspace{-2.5mu}*$)} shows that $\da{\mu}\delta_Z = \pi$, and thus $\pi(\delta_Z^{-1}\da{\varphi}\delta_Z) = \da{\mu}\da{\varphi}\delta_Z
= \da{\mu}\delta_Z = \pi$. As $\pi$ is an $\MCM$-cover, we conclude that $\delta_Z^{-1}\da{\varphi}\delta_Z$, and therefore also $\da{\varphi}$, is an automorphism. It follows that $\dda{\varphi}$ is an automorphism of $\ddda{Z}$, and finally that \smash{$\varphi = \delta_{\da{Z}}^{-1}\dda{\varphi}\delta_{\da{Z}}$} is an automorphism of $\da{Z}$.
\end{proof}

\section{Examples}
\label{sec:Examples}

We compute the $\MCM$-envelope of some specific modules. We begin with a characterization of modules with trivial $\MCM$-envelope.

\begin{prp}
  \label{prp:trivial-MCM-envelope}
  For a finitely generated $R$-module $M$, one has $\operatorname{dim}_RM<d$ if and only if 
  the zero map $M \to 0$ is an $\MCM$-envelope of $M$.
\end{prp}

\begin{proof}
  If $\operatorname{dim}_RM<d$ then \cite[Cor.~3.5.11(a)]{BruHer} shows that $\Hom{M}{\dual}=0$. It follows that every homomorphism $\varphi \colon M \to X$ with $X \in \MCM$ is zero. Indeed, since $\dual$ cogenerates the category $\MCM$, there exists a monomorphism $\iota \colon X \to \dual^n$ for some natural number $n$. As $\Hom{M}{\dual}=0$, the homomorphism $\iota\varphi \colon M \to \dual^n$ must be zero, and thus $\varphi=0$ since $\iota$ is injective. Since every homomorphism from $M$ to a maximal CM module is zero, the zero map $M \to 0$ is an $\MCM$-envelope of $M$.
  
  Conversely, if $M \to 0$ is an $\MCM$-(pre)envelope then, since $\dual$ is in $\MCM$, every homomorphism $\varphi \colon M \to \dual$ factors through $0$, and hence $\varphi=0$. Thus $\Hom{M}{\dual}=0$, and it follows from \cite[Cor.~3.5.11(b)]{BruHer} that one can not have $\operatorname{dim}_RM=d$; so $\operatorname{dim}_RM<d$.
\end{proof}

Next we give a somewhat ``general'' example.

\begin{exa}
  \label{exa:MCM-envelope-biduality}
  Let $M$ be a finitely generated $R$-module. If $\da{M}$ is maximal CM, then the identity homomorphism \smash{$\pi = 1_{\da{M}} \colon \da{M} \to \da{M}$} is an $\MCM$-cover of $\da{M}$. The proof of Theorem~A shows that the homomorphism \smash{$\mu = \da{\pi}\delta_M  = \delta_M$}, i.e.~the biduality homomorphism $\delta_M \colon M \to M^{\dagger\dagger}$, is an $\MCM$-envelope $M$.
\end{exa}

Here is a concrete application of the example above.

\begin{exa}
  \label{exa:ideal}
   Let $M$ be a submodule of a maximal CM $R$-module $X$ with the property that  $\operatorname{dim}_R(X/M)<d-1$. For example, $M=\a$ could be an ideal in $X=R$ with \mbox{$\operatorname{height}_R(\a)>1$}; see \cite[Cor.~2.1.4]{BruHer}. Or $M$ could be the submodule $M=(f_1,f_2,\ldots)X$ where $f_1,f_2,\ldots$ is an $X$-regular sequence of length at least two. We claim that, in this case, the inclusion map $\iota \colon M \hookrightarrow X$ is an $\MCM$-envelope of $M$.
   
   To see why, apply the functor $\da{(-)}$ to the short exact sequence \smash{$0 \to M \stackrel{\iota}{\to} X \to X/M \to 0$} to get the exact sequence
  \begin{displaymath}
    0 \longrightarrow \da{(X/M)} \longrightarrow \da{X} \stackrel{\da{\iota}}{\longrightarrow} \da{M} \longrightarrow \Ext{1}{X/M}{\dual}\;.
  \end{displaymath}
  Since \mbox{$d-\operatorname{dim}_R(X/M)>1$} it follows from \cite[Cor.~3.5.11(a)]{BruHer} that $\Hom{X/M}{\dual}=0$ and $\Ext{1}{X/M}{\dual}=0$. Hence the sequence displayed above shows that $\da{\iota}$ is an isomorphism and, in particular, $\da{M} \cong \da{X}$ is maximal CM. Thus \exaref{MCM-envelope-biduality} shows that the biduality homomorphism $\delta_M \colon M \to \dda{M}$ is an $\MCM$-envelope of $M$. It remains to argue that $\delta_M$ can be identified with $\iota \colon M \hookrightarrow X$; however, this follows from the commutative diagram:
  \begin{displaymath}
    \xymatrix{
       M \ar[r]^-{\iota} \ar[d]_-{\delta_M} & X \ar[d]_-{\cong}^-{\delta_X} \\
       \dda{M} \ar[r]_-{\cong}^-{\dda{\iota}} & \dda{X} 
    }
  \end{displaymath}
  Indeed, $\delta_X$ is an isomorphism as $X \in \MCM$, and $\dda{\iota} = \da{(\da{\iota})}$ is an isomorphism as $\da{\iota}$ is so.
\end{exa}

\begin{rmk}
  For a special $\MCM$-precover $\pi \colon X \to M$ of a finitely generated module $M$, the kernel $\Ker{\pi}$ has finite injective dimension, and hence one has $\Ext{i}{X}{\Ker{\pi}}=0$ for every $X \in \MCM$ and every $i>0$\,---\,not just for $i=1$. A similar phenomenon does not occur for special $\MCM$-preenvelopes. Indeed, if in \exaref{ideal} one has e.g.~$\operatorname{dim}_R(X/M)=d-2$, then $\Coker{\mspace{1mu}\iota} = X/M$ satisfies $\Ext{2}{X/M}{\dual} \neq 0$ by \cite[Cor.~3.5.11(b)]{BruHer}.
\end{rmk}

\section{$\MCM$-envelopes with the unique lifting property}
\label{sec:ulp}

If $\mu \colon M \to X$ is an $\MCM$-preenvelope of a finitely generated $R$-module $M$, then the induced homomorphism $\Hom{\mu}{Y} \colon \Hom{X}{Y} \to \Hom{M}{Y}$ is surjective for every $Y \in \MCM$; see~\dfnref{preenvelope}. If $\Hom{\mu}{Y}$ is an isomorphism for every $Y \in \MCM$, then we say that the $\MCM$-preenvelope $\mu$ has the \emph{unique lifting property}. Indeed, in this case, there exists for every homomorphism $\nu \colon M \to Y$ with $Y \in \MCM$ a unique homomorphism $\varphi \colon X \to Y$ that makes the following diagram commute:
\begin{displaymath}
  \xymatrix{
    M \ar[r]^-{\mu} \ar[d]_-{\nu} & X \ar@{-->}[dl]^-{\varphi} \\
    Y & {}
   }
\end{displaymath}
Note that an $\MCM$-preenvelope $\mu \colon M \to X$ with the unique lifting property must necessarily be an $\MCM$-envelope. Indeed, the only endomorphism $\varphi$ of $X$ with $\varphi\mu=\mu$ is $\varphi=1_X$.

\enlargethispage{1ex}

\begin{lem}
  \label{lem:depth-of-dagger-dual}
  For any finitely generated $R$-module $M$, one has $\depth{(\da{M})} \geqslant \min\{d,2\}$.
\end{lem}

\begin{proof}
  Take an exact sequence $L_1 \to L_0 \to M \to 0$ where $L_0$ and $L_1$ are finitely generated and free. Since the functor $\da{(-)}=\Hom{-}{\dual}$ is left exact, we get an exact sequence, $0 \to \da{M} \to \da{L_0} \to \da{L_1} \to C \to 0$, where $C$ is the cokernel of the homomorphism $\da{L_0} \to \da{L_1}$. Since the modules \smash{$\da{L_0}$} and \smash{$\da{L_1}$} are maximal CM, \lemref{depth} yields
  \begin{displaymath}
    \depth{(\da{M})} \geqslant \min\{d,\depth{C}+2\} \geqslant \min\{d,2\}\;. \qedhere
  \end{displaymath}
\end{proof}

\begin{proof}[Proof of Theorem~B]
  \proofofimp{i}{ii} Consider an exact sequence of finitely generated modules
  \begin{displaymath}
    0 \longrightarrow K \longrightarrow L_1 \stackrel{\alpha}{\longrightarrow} L_0 \longrightarrow N \longrightarrow 0\;,
  \end{displaymath}
where $L_0$ and $L_1$ are free and $K = \Ker{\alpha }$. From 
\cite[Prop.~1.2.9]{BruHer} (last inequality) one gets 
\begin{displaymath}
  \tag{\text{$*$}}
  \depth{N} \geqslant \depth{K}-2\;.
\end{displaymath}
Set $C=\Coker{(\da{\alpha})}$ and consider the exact sequence \smash{$\da{L_0} \stackrel{\da{\alpha}}{\longrightarrow} \da{L_1} \longrightarrow C \longrightarrow 0$}. As the functor $\da{(-)}$ is left exact, we get a commutative diagram with exact rows,
\begin{displaymath}
  \xymatrix{
     0 \ar[r] & K \ar[r] & L_1 \ar[d]_-{\cong}^-{\delta_{L_1}} \ar[r]^-{\alpha} & 
     L_0 \ar[d]_-{\cong}^-{\delta_{L_0}}
     \\
     0 \ar[r] & \da{C} \ar[r] & \dda{L_1} \ar[r]^-{\dda{\alpha}} & \dda{L_0}
  }
\end{displaymath}
which shows that $K \cong \da{C}$, since $\delta_{L_0}$ and $\delta_{L_1}$ are isomorphisms. By the assumption \eqclbl{i}, the module $K$ is therefore maximal CM, and hence the inequality ($*$) yields $\depth{N} \geqslant d-2$. As this holds for every finitely generated $R$-module $N$, it holds in particular for the residue field $N=k$. We get $0=\depth{k} \geqslant d-2$, and thus $d \leqslant 2$.

\proofofimp{ii}{iii} If $d \leqslant 2$, then \lemref{depth-of-dagger-dual} shows that for every finitely generated $R$-module $M$, the module $\da{M}$ is maximal CM, and hence so is $\dda{M}$. Thus $\mathrm{F}=\dda{(-)}$ is a functor $\mod \to \MCM$, which we claim is a left adjoint of the inclusion $\mathrm{G} \colon \MCM \to \mod$. For each finitely generated $R$-module $M$ and each maximal CM $R$-module $X$, the homomorphism
\begin{displaymath}
  \xymatrix@C=6.5pc{
    \Hom{\mathrm{F}M}{X} = \Hom{\dda{M}}{X} \ar[r]^-{\varphi_{M,X}\,=\,\Hom{\delta_M}{X}} & \Hom{M}{X} = \Hom{M}{\mathrm{G}X}
  }
\end{displaymath}
is evidently natural in $M$ and $X$; and it is surjective since the biduality map $\delta_M \colon M \to M^{\dagger\dagger}$ is an $\MCM$-preenvelope of $M$ by \exaref{MCM-envelope-biduality}. It remains to see that $\Hom{\delta_M}{X}$ is in\-jec\-tive. To this end, let $\mu \colon \dda{M} \to X$ be a homomorphism with $\mu \delta_M = \Hom{\delta_M}{X}(\mu) = 0$. It follows that \smash{$\da{\delta_M}\da{\mu} = \da{(\mu \delta_M)} = 0$}. As $\da{M}$ is maximal CM, the biduality map \smash{$\delta_{\da{M}}$} is an isomorphism, and hence so is 
\smash{$\da{\delta_M}$} by \lemref{split}. Since \smash{$\da{\delta_M}\da{\mu} = 0$} we conclude that $\da{\mu} = 0$. Thus $\dda{\mu}=\da{(\da{\mu})} = 0$ and consequently $\mu= \delta_X^{-1}\dda{\mu}\delta_{\dda{M}}=0$, as desired.

\proofofimp{iii}{iv} Let $\mathrm{F} \colon \mod \to \MCM$ be a left adjoint of the inclusion $\mathrm{G} \colon \MCM \to \mod$. For every finitely generated $R$-module $M$, the unit of adjunction $\eta_M \colon M \to \mathrm{G}\mathrm{F} M$ induces, for every maximal CM $R$-module $Y$, an isomorphism:
\begin{displaymath}
  \varphi_{M,Y} \colon \Hom{\mathrm{F} M}{Y} \stackrel{\cong}{\longrightarrow} \Hom{M}{\mathrm{G} Y} 
  \quad \text{ given by } \quad \alpha \longmapsto \mathrm{G}(\alpha)\eta_M\;;
\end{displaymath}
see \cite[IV.1~Thm.~1]{Mac}. If we suppress the inclusion functor $\mathrm{G}$ and set $X= \mathrm{G}\mathrm{F} M = \mathrm{F} M$, which is maximal CM by the assumption on $\mathrm{F}$, we see that unit of adjunction $\eta_M \colon M \to X$ has the property that the map
\begin{displaymath}
  \Hom{X}{Y} \stackrel{\cong}{\longrightarrow} \Hom{M}{Y} 
  \quad \text{ given by } \quad \alpha \longmapsto \alpha\eta_M = \Hom{\eta_M}{Y}(\alpha)
\end{displaymath}
is an isomorphism. Thus, $\eta_M$ is an $\MCM$-envelope of $M$ with the unique lifting property.

\proofofimp{iv}{i} Let $M$ be a finitely generated $R$-module. By assumption, $M$ has an $\MCM$-envelope $\mu \colon M \to X$ with the unique lifting property. Since $\dual$ is maximal CM, the homomorphism $\da{\mu} \colon \da{X} \to \da{M}$ is an isomorphism, and as $\da{X}$ is maximal CM, so is $\da{M}$.
\end{proof}

\section{Cosyzygies with respect to $\MCM$}
\label{sec:cosyz}

Let $\mathcal{A}$ be a full subcategory of an abelian category $\mathcal{M}$ (for example, $\mathcal{M}=\mod$ and $\mathcal{A}=\MCM$). 

Assume that every object in $\mathcal{M}$ has an $\mathcal{A}$-precover. In this case, every $M \in \mathcal{M}$ admits a \emph{proper $\mathcal{A}$-resolution} i.e.~a, not necessarily exact, complex \mbox{$\mathbb{A} = \cdots \to A_1 \to A_0 \to M \to 0$} with $A_i \in \mathcal{A}$ such that the sequence $\Hom[\mathcal{M}]{A}{\mathbb{A}}$ is exact for every $A \in \mathcal{A}$. Such a resolution is constructed recursively as follows: Take an $\mathcal{A}$-precover $\pi_0 \colon A_0 \to M$ of $M$ and set $K_1 = \Ker{\pi_0}$; take an $\mathcal{A}$-precover $\pi_1 \colon A_1 \to K_1$ of $K_1$ and set $K_2 = \Ker{\pi_1}$; etc. The object $K_n$ is denoted by $\Syz{n}{M}$ and it is called the \emph{$n^\mathrm{th}$ syzygy of $M$ with respect to $\mathcal{A}$}. A given object $M \in \mathcal{M}$ has, typically, many different $\mathcal{A}$-precovers and proper $\mathcal{A}$-resolutions, so $\Syz{n}{M}$ is not uniquely determined by $M$; but it almost is: The version of Schanuel's lemma found in \cite[Lem.~2.2]{EJO-01} shows that if $K_n$ and $\bar{K}_n$ are both $n^\mathrm{th}$ syzygies of $M$ with respect to $\mathcal{A}$, then there exist $A,\bar{A} \in \mathcal{A}$ such that $K_n \oplus \bar{A} \cong \bar{K}_n \oplus A$. In particular, if $\mathcal{A}$ is closed under direct summands (as is the case if $\mathcal{A} = \MCM$), then $K_n$ belongs to $\mathcal{A}$ if and only if $\bar{K}_n$ belongs to $\mathcal{A}$; and thus it makes sense to ask if $\Syz{n}{M}$ belongs to $\mathcal{A}$.

If every object in $\mathcal{M}$ admits an $\mathcal{A}$-cover, then $\pi_0,\pi_1,\ldots$ in the construction above can be chosen to be $\mathcal{A}$-covers, and the resulting proper $\mathcal{A}$-resolution is then called a \emph{minimal proper $\mathcal{A}$-resolution} of $M$. In this case, $K_n$ is called the \emph{minimal $n^\mathrm{th}$ syzygy of $M$ with respect to $\mathcal{A}$}, and it is denoted by $\mSyz{n}{M}$ (small ``$\mathrm{s}$'' instead of capital ``$\mathrm{S}$''). Since an $\mathcal{A}$-cover (of a given object in $\mathcal{M}$) is unique up to isomorphism, see Xu \cite[Thm.~1.2.6]{xu}, the object $\mSyz{n}{M}$ is uniquely determined, up to isomorphism, by $M$.

Dually, if every $M \in \mathcal{M}$ has an $\mathcal{A}$-preenvelope (respectively, $\mathcal{A}$-envelope), then a \emph{proper $\mathcal{A}$-coresolution} (respectively, \emph{minimal proper $\mathcal{A}$-coresolution}) \mbox{$0 \to M \to A^0 \to A^1 \to \cdots$} 
can always be constructed as follows: Take an $\mathcal{A}$-preenvelope (respectively, $\mathcal{A}$-envelope) $\mu^0 \colon M \to A^0$ of $M$ and set $C^1 = \Coker{\mu^0}$; take an $\mathcal{A}$-preenvelope (respectively, $\mathcal{A}$-envelope) $\mu^1 \colon C^1 \to A^1$ of $C^1$ and set $C^2 = \Coker{\mu^1}$;~etc. The object $C^n$ is called the \emph{$n^\mathrm{th}$ cosyzygy of $M$ with respect to $\mathcal{A}$} (respectively, the \emph{minimal $n^\mathrm{th}$ cosyzygy of $M$ with respect to $\mathcal{A}$}) and it is denoted by $\Cosyz{n}{M}$ (respectively, $\mCosyz{n}{M}$). The object $\mCosyz{n}{M}$ is uniquely determined, up to isomorphism, by $M$. The object $\Cosyz{n}{M}$ is almost uniquely determined by $M$ in the sense that if $C^n$ and $\bar{C}^n$ are both $n^\mathrm{th}$ cosyzygies of $M$ with respect to $\mathcal{A}$, then there exist $A,\bar{A} \in \mathcal{A}$ such that $C^n \oplus \bar{A} \cong \bar{C}^n \oplus A$. Thus, if $\mathcal{A}$ is closed under direct summands, then it makes sense to ask if $\Cosyz{n}{M}$ belongs to $\mathcal{A}$.

We supplement the definitions above by setting $\Syz{0}{M} = \mSyz{0}{M} = M$, and similarly $\Cosyz{0}{M} = \mCosyz{0}{M} = M$.

\begin{exa}
\label{exa:minimal-free-resolution}
Let $(A,\n,\ell)$ be any local ring and let $\mathcal{F}$ be the class of finitely generated free $A$-modules. Every finitely generated $A$-module $M$ has an $\mathcal{F}$-cover; to construct it one takes a minimal set $x_1,\ldots,x_b$ of generators of $M$ (here $b=\beta_0^A(M)$ is the zero'th Betti number of $M$) and then defines $A^b \twoheadrightarrow M$ by $e_i \mapsto x_i$; see \cite[Thm.~5.3.3]{rha}. A minimal proper $\mathcal{F}$-resolution 
$\cdots \to F_1 \to F_0 \to M \to 0$ of a finitely generated $A$-module $M$ is nothing but a \emph{minimal free resolution} of $M$ in the classical sense, that is, where each homomorphism $F_{n} \to F_{n-1}$ becomes zero when tensored with the residue field $\ell$ of $A$.
\end{exa}

\begin{prp}
  \label{prp:second-syzygy}
  Let $M$ be a finitely generated $R$-module such that $\da{M}$ is maximal CM. Then the second cosyzygy, $\Cosyz[\MCM]{2}{M}$, of $M$ with respect to $\MCM$ is maximal CM.
\end{prp}

\begin{proof}
  By \exaref{MCM-envelope-biduality} the biduality homomorphism $\delta_{M} \colon M \to \dda{M}$ is an $\MCM$-envelope of $M$. Set $C^1 = \mCosyz[\MCM]{1}{M} = \Coker{\delta_M}$. The exact sequence \smash{$M \stackrel{\delta_M}{\longrightarrow} \dda{M} \longrightarrow C^1 \longrightarrow 0$} induces, by application of the left exact functor $\da{(-)}$, an exact sequence
  \begin{displaymath}
    \xymatrix@C=1.5pc{
       0 \ar[r] & \da{(C^1)} \ar[r] & \ddda{M} \ar[r]^-{\da{\delta_M}} & \da{M} 
     }.
  \end{displaymath}
  As $\da{M}$ is maximal CM, the biduality homomorphism $\delta_{\da{M}}$ is an isomorphism, and hence so is \smash{$\da{\delta_M}$} by \lemref{split}. It follows that $\Hom{C^1}{\dual}=\da{(C^1)}=0$, so \cite[Cor.~3.5.11(b)]{BruHer} implies that $\operatorname{dim}_R(C^1)<d$. Thus \prpref{trivial-MCM-envelope} shows that $C^1 \to 0$ is an $\MCM$-envelope of $C^1$, and therefore the minimal second cosyzygy of $M$ with respect to $\MCM$ is zero:  
\begin{displaymath}
  \mCosyz[\MCM]{2}{M} \,=\, \mCosyz[\MCM]{1}{C^1} \,=\, \Coker{(C^1 \to 0)} \,=\,0\;.
\end{displaymath}
Hence any second cosyzygy of $M$ with respect to $\MCM$ must be maximal CM.
\end{proof}

We now prove Theorem~C from the Introduction.

\begin{proof}[Proof of Theorem~C]
  First note, that if $X$ is a maximal CM $R$-module, then $\Cosyz[\MCM]{i}{X}$ is clearly maximal CM for every $i \geqslant 0$. If $n \geqslant d$, then the
  $n^\mathrm{th}$ cosyzygy of $M$ is an $(n-d)^\mathrm{th}$ cosyzygy of $\Cosyz[\MCM]{d}{M}$, that is, 
\begin{displaymath}  
  \Cosyz[\MCM]{n}{M} \,=\, \Cosyz[\MCM]{n-d}{\Cosyz[\MCM]{d}{M}}\;;
\end{displaymath}  
so it suffices to argue that $\Cosyz[\MCM]{d}{M}$ is maximal CM.

If $d=0$, then certainly $\Cosyz[\MCM]{0}{M}=M$ is maximal CM, since every finitely generated $R$-module is maximal CM over an artinian ring.
    
Assume that $d=1$. By Theorem~A we can take a special $\MCM$-preenvelope $\mu \colon M \to X$ of $M$. We must show that $C^1 = \Cosyz[\MCM]{1}{M} = \Coker{\mu}$ is maximal CM. By definition, we have $\Ext{1}{C^1}{Y}=0$ for all $Y \in \MCM$, in particular, $\Ext{1}{C^1}{\dual}=0$. Since $\dual$ has injective dimension $d=1$, we also have $\Ext{i}{-}{\dual}=0$ for all $i>1$, and consequently, $\Ext{i}{C^1}{\dual}=0$ for all $i>0$. Thus $C^1$ is maximal CM.

Finally, assume that $d \geqslant 2$. Let $0 \to M \to X^0 \to \cdots \to X^{d-3} \to C^{d-2} \to 0$ be part of a proper $\MCM$-coresolution of $M$, where $C^{d-2} = \Cosyz[\MCM]{d-2}{M}$. In the case $d=2$, this just means that we consider the module $C^{0} = \Cosyz[\MCM]{0}{M}=M$. Since the module $\dual$ is maximal CM, the sequence
\begin{displaymath}  
  0 \longrightarrow \da{(C^{d-2})} \longrightarrow \da{(X^{d-3})} \longrightarrow \cdots \longrightarrow \da{(X^0)} \longrightarrow \da{M} \longrightarrow 0
\end{displaymath}  
is exact. Now \lemref[Lemmas~]{depth} and \lemref[]{depth-of-dagger-dual} yield $\depth{\da{(C^{d-2})}} \geqslant \min\{d,\depth{\da{M}}+d-2\} = d$, so
$\da{(C^{d-2})} = \da{(\Cosyz[\MCM]{d-2}{M})}$ is maximal CM. \prpref{second-syzygy} now yield that 
\begin{displaymath}  
  \Cosyz[\MCM]{d}{M} \,=\, \Cosyz[\MCM]{2}{\Cosyz[\MCM]{d-2}{M}}
\end{displaymath}  
is  maximal CM, as desired.
\end{proof}

Dutta \cite{Dutta} shows that if $R$ is not regular, then no syzygy in the minimal free resolution of the residue field $k$ (see \exaref{minimal-free-resolution}) can contain a non-zero free direct summand. The following result has the same flavour, but its proof is easy. Actually, the proof of \cite[Prop. 2.6]{RTk06b} applies to prove \prpref{no-free-summand} as well, but since it is so short, we repeat it here.

\begin{prp}
  \label{prp:no-free-summand}
  Assume that every finitely generated $R$-module has an $\MCM$-envelope (by Theorem~A, this is the case if $R$ is henselian). Let $M$ be a finitely generated $R$-module and let $n \geqslant 1$ be an integer. The minimal $n^\mathrm{th}$ cosyzygy, $\mCosyz[\MCM]{n}{M}$, of $M$ with respect to $\MCM$ contains no non-zero free direct summand.
\end{prp}

\begin{proof}
  It suffices to consider the case $n=1$. Let $\mu \colon M \to X$ be an $\MCM$-envelope of $M$, set $C=\mCosyz[\MCM]{1}{M} = \Coker{\mu}$, and write write $\pi \colon X \twoheadrightarrow C$ for the canonical homomorphism. Let $F$ be a free direct summand in $C$ and denote by $\rho \colon C \twoheadrightarrow F$ the projection onto this direct summand. We have a commutative diagram,
\begin{displaymath}
  \xymatrix{
    {} & 
    M \ar[r]^-{\mu} \ar[d]^-{\mu_0} & 
    X \ar[r]^-{\pi} \ar@{=}[d] & 
    C \ar@{->>}[d]^-{\rho} \ar[r] & 0\;\phantom{,}
    \\
    0 \ar[r] & 
    K \ar[r]^-{\iota} & 
    X \ar[r]^-{\rho\pi} & 
    F \ar[r] & 0\;,
  }
\end{displaymath}
where $\iota \colon K \to X$ is the kernel of $\rho\pi$, and $\mu_0$ is the corestriction of $\mu$ to $K$. Since $F$ is free, the lower short exact sequence splits, so $\iota$ has a left inverse $\sigma \colon X \to K$. The endomorphism $\iota\sigma$ of $X$ satisfies $\iota\sigma\mu = \iota\sigma\iota\mu_0 = \iota\mu_0 = \mu$, and since $\mu$ is an envelope, we conclude that $\iota\sigma$ is an automorphism. In particular, $\iota$ is surjective, and hence $F$ is zero.
\end{proof}

\def\cprime{$'$}
  \providecommand{\arxiv}[2][AC]{\mbox{\href{http://arxiv.org/abs/#2}{\sf
  arXiv:#2 [math.#1]}}}
  \providecommand{\oldarxiv}[2][AC]{\mbox{\href{http://arxiv.org/abs/math/#2}{\sf
  arXiv:math/#2
  [math.#1]}}}\providecommand{\MR}[1]{\mbox{\href{http://www.ams.org/mathscinet-getitem?mr=#1}{#1}}}
  \renewcommand{\MR}[1]{\mbox{\href{http://www.ams.org/mathscinet-getitem?mr=#1}{#1}}}
\providecommand{\bysame}{\leavevmode\hbox to3em{\hrulefill}\thinspace}
\providecommand{\MR}{\relax\ifhmode\unskip\space\fi MR }
\providecommand{\MRhref}[2]{%
  \href{http://www.ams.org/mathscinet-getitem?mr=#1}{#2}
}
\providecommand{\href}[2]{#2}

\end{document}